\documentclass[a4paper,12pt]{amsart}
\usepackage[cp1251]{inputenc}
\usepackage[T2A]{fontenc}
\usepackage[english]{babel}
\def\No{№}
\usepackage{amsfonts}
\usepackage{amssymb}
\usepackage{latexsym}
\small\large
\newtheorem{teo}{{\bf Theorem}}
\newtheorem{lemma}{\bf Lemma}

\newtheorem{pre}{\bf }

\title{Sharply doubly transitive groups with saturation conditions}

\author{Anatoliy I.Sozutov, Evgeny B. Durakov }

\begin{document}

\maketitle

{\small\bf Annotation:} {\small A number of conditions were found under which
a sharply doubly transitive permutation group has an abelian normal
divider. }

\smallskip


\bigskip

Recall that the permutation group $G$ of the set $F$ ($ | F | \geq k $) is called
{\it is exactly $ k $ -transitive} on $ F$ if for any two ordered
sets $ (\alpha_1, \, ..., \, \alpha_k) $ and
$ (\beta_1, \, ..., \, \beta_k) $ elements from $ F $ such that
$ \alpha_i \ne \alpha_j $ and $ \beta_i \ne \beta_j $ for $ i \ne j $,
there is exactly one element of the group $ G $ taking
$ \alpha_i $ to $ \beta_i $ ($ i = 1, \, ..., \, k $).

As K. Jordan \cite{J} proved, in a finite sharply doubly transitive group
$ T $ regular permutations (moving each element of the set $ F $)
together with the identity substitution constitute an abelian normal subgroup.
Is Jordan's theorem true for infinite groups,
in the general case is unknown (see, for example,
 \cite [questions 11.52, 12.48] {Kou}).
To particular solutions of this central question in the theory of near-fields and
near-domains \cite{W} dozens of works by famous authors are devoted.
We especially note that in 2014 in \cite{RST,TenZ} exactly
doubly transitive groups without abelian normal subgroups, characteristics 2.
\cite{RipsTent} and \cite{AndreTent} constructed examples of exactly 2-transitive groups and
simple exactly 2-transitive groups, respectively, that have characteristic 0 and do not contain non-trivial Abelian normal subgroups. And in 2008 \cite{Turk} studied the question for exactly 2-transitive groups of characteristic 3.
 For other characteristics the question remains open.
Sharply 2-transitive groups are closely related to
algebraic structures such as
near-fields, near-areas,
$ KT $ -fields (Kerby-Tits fields), projective planes, etc. (see \cite [Ch. V] {W},
\cite [chap. 20] {Hal}).

We continue to investigate infinite exactly doubly transitive
groups and related near-domains \cite {BDS, SD3}.
In this paper, a number of conditions are found under which a group has an abelian normal
divisor (see \cite [questions 11.52, 12.48] {Kou}),
and the corresponding near-domain \cite{W} is a near-field.
All necessary definitions are given in paragraph 1.
We only recall that the group $ G $ is saturated with groups of some set of finite groups
$ \mathfrak {X} $ if every finite subgroup of $ G $ is contained in
a subgroup of the group $ G $ isomorphic to some group from $ \mathfrak {X} $.

\smallskip

This work was supported by the Russian Foundation for Basic Research (project \No 19-01-00566 A) and Krasnoyarsk
mathematical center, funded by the Ministry of Education and Science of the Russian Federation in the framework of
activities for the creation and development of regional NMCs
(Agreement 075-02-2020-1534 / 1).

\smallskip

\begin{teo}
Sharply doubly transitive group $ T $ saturated with finite Frobenius groups
of substitutions, the set $ F $ of odd characteristic has a regular abelian normal subgroup
and the almost domain $ F (+, \cdot) $ is a near-field if at least
one of the following properties of the stabilizer $ T_ \alpha $ of the point $ \alpha \in F $:
\begin {enumerate}
\item $ T_ \alpha $ --- Shunkov group;
\item $ T_ \alpha $ is a periodic group and $ T_\alpha $ does not contain conjugate
dense subgroups;
\item $ T_ \alpha $ is (locally)
a finite normal subgroup of order more than 2;
\item $ T_\alpha $ contains a finite element $ a $ of prime odd order
not equal to five.
\end {enumerate}
\label{t1}
\end{teo}

\begin{teo}
Let in a sharply triply transitive permutation group $ G $
odd characteristic point stabilizer
has a regular abelian normal subgroup
and each of its finite subgroups is contained in a finite unsolvable
subgroup of the group $ G $. Then $ G $ is locally finite.
\label{t2}
\end{teo}

\section{Results and concepts used}

Let $ T $ be a sharply doubly transitive group
permutations of the set $ F $.
According to M. Hall's theorem \cite{W}, we assume that $ F $ is almost a domain
with operations of addition $ + $ and
multiplication $ \cdot $, and $ T $ is its affine transformation group
$ x \to a + bx $ $ (b \ne 0) $, and $ T_{\alpha} \simeq F^* (\cdot) $ is the stabilizer of the point $ \alpha \in F $.
Let $ T_ \alpha $ be the stabilizer of the point $ \alpha \in F $,
$ J $ denotes the set of involutions of the group $ T $. For each involution $ k \in
J $ by $ N_k $ we denote the set $ kJ = \{kj \mid j \in J \} $.
The result of a substitution action
$ t \in T $ to an element (point) $ \gamma \in F $ is denoted by $ \gamma^t $.
The following statements are true.

\begin{pre}
$(T,T_{\alpha})$ --- a Frobenius pair, i.e.
$T_{\alpha}\cap T^t_{\alpha}=1$ for anyone $t\in T\setminus T_{\alpha}$.
\label{p1}
\end{pre}

\begin{pre}
The group $ T $ contains involutions and all involutions in $ T $ are conjugate.
The set $ J \setminus T_\alpha $ is nonempty and
$ T_{\alpha} $ is transitive on $ J \setminus T_\alpha $.
Product of two different involutions from $ T $
is a regular substitution, i.e. acting on $ F $ without fixed points.
\label{p2}
\end{pre}

\begin{pre}
If $ T_{\alpha} $ contains an involution $ j $, then it is unique in $ T_{\alpha} $,
$ T = T_\alpha N_j $ and
$ T $ acts by conjugation on the set $ J $ of all its involutions
exactly twice transitively. All products $ kv $, where $ k, v \in J $, $ k \ne v $
in the group $ T $ are conjugate and have the same or a simple odd order
$ p $, or infinite order. In the first case $ {\rm Char} \, T = p $,
in the second case $ {\rm Char} \, T = 0 $ (if there are no involutions in $ T_{\alpha} $,
then $ {\rm Char} \, T = 2 $ by definition).
\label{p3}
\end{pre}

\begin{pre} The following result was proved in \cite{BDS}
Let $ {\rm Char} \, T = p> 2 $,
$ b $ is strictly real with respect to
$ j $ is an element from $ T \smallsetminus T_ {\alpha} $, $ A = C_T (b) $ and $ V = N_T (A) $.
Then: 

1) the subgroup $ A $ is periodic, abelian is inverted by the involution $ j $
and is strongly isolated in $ T $; 

2)the subgroup $ V $ acts exactly twice transitively on the orbit $ \Delta = \alpha^A $,
moreover, $ A $ is an elementary abelian regular normal subgroup of $ V $,
$ H = V \cap T_{\alpha} $ is a point stabilizer and $ V = A \lambda H $.
 If in addition $ | \bigcap_{x \in T_{\alpha}} H^ x |> 2, $
then $ A $ is a regular abelian normal subgroup of $ T $,
and near-domain
$ F (+, \cdot) $ is a near-field.
\label{p7}
\end{pre}

The next proposition shows that the saturation condition
finite Frobenius groups partially holds in an arbitrary
exactly doubly transitive group of odd characteristic.

\begin{pre}
When $ {\rm Char} \, T = p> 2 $, each
dihedral subgroup of $ T $
is contained in a finite Frobenius subgroup of $ T $ with kernel of order $ p $
and a cyclic complement of order $ p-1 $.
\label{p4}
\end{pre}

Recall that $ (a, b) $ - {\it finiteness condition} in the group $ G $ means that $ a $ and $ b $ ---
its nontrivial elements, and in the group $ G $ all subgroups
$ \langle a, b ^ x \rangle $ ($ x \in G $). The elements $ a $ and $ b $ are called
{\it is generalized finite}, and if $ a = b $, then $ a $ is called {\it finite}
elements in the group $ G $.
The following proposition was proved in \cite [Lemmas 1, 2 and Theorem] {SD3}:

\begin{pre}
Let $ T $ be a sharply doubly transitive group
characteristics $ {\rm Char} \, T \ne 2 $ contains elements $ a, b $ with $ (a, b) $ - condition
limbs.
Then
\begin {enumerate}
\item If for any $ x \in T $ the subgroup $ \langle a^x \rangle $ and $ \langle b \rangle $
are not incident, then at least one of the elements $ a, b $
belongs to the stabilizer of some point of the near-domain $ F $;
\item If $ | a | = | b |> 2 $, $ T $ has a regular abelian normal
subgroup;
\item If $ | a | \cdot | b | = 2k> 4 $, $ T $ has a regular abelian normal
subgroup.
\end {enumerate}
\label{p5}
\end{pre}

\begin {pre}\cite [Theorem 2] {C}
If the group $ T $ has
is a finite element of order $> 2 $, then $ T $ has a regular
abelian normal subgroup and near-domain
$ F $ is a near-field.
\label{p6}
\end{pre}

\begin{pre}
If $ T $ contains a locally finite subgroup containing a regular
substitution and intersecting $ T_ \alpha $ by a normal subgroup in it,
consisting of more than two elements, then $ T $ has
a regular abelian normal subgroup, and the near-domain
$ F (+, \cdot) $ is
near-field.
\end{pre}

\begin{pre}
If $ {\rm Char \,} T \ne 2 $, $ T $ contains the Frobenius group
$ V $ with involution and complement $ H = V \cap T_\alpha $,
where $ T_\alpha $ is the stabilizer of the point $ \alpha $, and in $ H $ there is a normal in
$ T_\alpha $ is a subgroup of order $> 2 $,
then the group $ T $ has a regular abelian normal subgroup.
\label{st2}
\end{pre}

\begin{pre}
In the complement of a finite Frobenius group, each cyclic subgroup
of prime order $ q> 5 $ is normal.
\label{p9}
\end{pre}

\section{Proof of Theorem 1}

Let the group $ T $ and the near-domain $ F $ satisfy the conditions of Theorem 1.
The set of Frobenius subgroups of the group $ T $ containing
finite subgroup $ L $, we denote by $ \mathfrak {X} (L) $.
For any Frobenius group $ M \leq T $, we agree to its kernel
denote by $ F_M $, and the complement suitable for the meaning of the text ---
through $ H_M $.
The notation $ L \leq H_M $ often used in the following text
for the subgroup $ L \leq M \in \mathfrak {X} (L) $ will be
mean that $ L $ is contained in some complement $ H_M $
group $ M $.

\smallskip

{\bf Comment.} {\it
It is well known that a near-domain $ F (+, \cdot) $ is a near-field if and only if
when $ T $ has a regular abelian normal subgroup.
Therefore, to prove Theorem 1
it suffices to prove the existence in the infinite group $ T $ of a regular abelian
normal subgroup.}

\begin{lemma}
We can assume
what ${\rm Char}\, T\ne 3$.
\label{l1}
\end{lemma}

\begin{proof}
In the case of $ {\rm Char} \, T = 3 $, the group $ T $
has a regular abelian normal subgroup
and without the additional saturation condition
(see, for example, \cite [Lemma 2.7] {Preprint}. The lemma is proved.
\end{proof}

\begin{lemma}
When $ T_\alpha $ has
is a finite element of prime order $ q> 5 $, the theorem is true.
\label {l2}
\end {lemma}
\begin {proof}
 Let $ {\rm Char} \, T = p> 3 $, $ j $ be an involution from $ T_\alpha $
and $ a $ --- final
in $ T_\alpha $ is an element of prime order $ q> 5 $.
Let's choose an arbitrary element $ t \in T_\alpha $.
Due to the finiteness of the element $ a $ and item 3, the subgroup $ L_t = \langle a, a^t, j \rangle $ is finite,
and $ L_t \leq M \in \mathfrak {X} (K) $. In view of item 1 and the properties of finite Frobenius groups
$ M \cap T_\alpha \leq H_M $, $ F_M \cap T_ \alpha = 1 $ and $ F_M $ is an elementary abelian $ p $ -group,
consisting of all products $ jk $, where $ k \in J \cap M $.
According to the structure of complements in finite Frobenius groups
$ L_t = \langle a, a^t \rangle = \langle a \rangle $.
Since the element $ t \in T_\alpha $ is arbitrary, we conclude that
that the subgroup $ \langle a \rangle $ is normal in $ T_\alpha $.
Therefore, the subgroup $ L = \langle a, j \rangle $ is contained in the appendix
finite Frobenius group $ M ^ x \in \mathfrak {X} (A) $ for any $ x \in T_\alpha $,
and the set-theoretic union of the kernels of all such groups (for $ x \in T_\alpha $),
contains the set $ N_j $. From this and the equality $ T = T_ \alpha N_j $ (item 3) it follows that all subgroups
$ L_g = \langle a, a^g \rangle $ are finite for any $ g \in T $,
and by item 6 $ T $ possesses a regular
abelian normal subgroup. The lemma is proved.
\end{proof}

\begin{lemma}
If $T_\alpha$ is a Shunkov group and ${\rm Char}\,T\ne 2$,
then $T_\alpha$ has a local finite periodic part and
subgroup $\Omega_1(T_\alpha )$ generated by all elements of prime orders from
$T_\alpha$, there is a group of one of the following types:
1) (locally) cyclic group;
$\Omega_1(T_\alpha )=C\times L$, where $C$ --- (locally) cyclic $\{2,3\}'$
-- group, and $L\simeq SL_2(3)$; 3) $\Omega_1(T_\alpha )=С\times L$,
where $C$ is (locally) cyclic $\{2,3,5\}'$
-- group, $L\simeq SL_2(5)$. In any case, in $T_\alpha$
is a finite normal subgroup of order greater than 2.
\label{l3}
\end{lemma}
\begin{proof}
By virtue of the saturation condition and the condition $ {\rm Char} \, T \ne 2 $, each finite
the subgroup $ K $ of $ T_\alpha $ is contained in the complement of the finite group
Frobenius $ M \in \mathfrak {X} (K) $, and according to
To the proposition \cite [Proposition 11] {C1}, its subgroup $ \Omega_1 (K) $ has the structure indicated in the lemma.
By \cite [Theorem 1] {C}, the same structure has
subgroup $ \Omega_1 (T_ \alpha) $, and $ T_ \alpha $
possesses a local finite periodic part, we denote it by $ S $.
If $ S $
not 2-group, then
$ \Omega_1 (T_\alpha) $ obviously contains a finite normal subgroup of $ T_\alpha $ of order,
greater 2. If $ S $ is a 2-group, then by \cite [Theorem 2] {CCC} either a quasicyclic
group, either locally quaternionic and also contains
finite normal in $ T_\alpha $
a subgroup of order $ 2^n $ for any $ n> 1 $. The lemma is proved.
\end{proof}

\begin{lemma}
The theorem is true if  $ T_\alpha $ contains a finite
normal subgroup
$ L $, $ | L |> 2 $.
\label {l4}
\end {lemma}
\begin {proof}
Since the involution of $ j $ in $ T_\alpha $ is unique, we can assume that
$ j \in L $. By the saturation condition $ M \in \mathfrak {X} (L) $, and
the subgroup $ M \cap T_ \alpha $ is strongly isolated in $ M $.
Therefore, the subgroup $ L $ is the complement of some finite group
Frobenius $ M \in \mathfrak {X} (L) $. From the normality of $ L $ to $ T_\alpha $
it follows that $ M^t \in \mathfrak {X} (L) $ and $ H_M = L $.
Hence it follows that
set-theoretic union of the kernels of all such groups (for $ t \in T_\alpha $),
contains the set $ N_j $. Let $ a $ be an element of order greater than 2 from $ L $
(since $ | L |> 2 $, such an element exists by virtue of item 3).
It follows from what has been proved that all subgroups
$ L_g = \langle a, a ^ g\rangle $ are finite for any $ g \in T $,
and by item 6 $ T $ possesses a regular
abelian normal subgroup. The lemma is proved.
\end{proof}

\begin{lemma}
When $ T_\alpha $ contains a finite element $ a $ of order $ 3 $,
the theorem is also true.
\label {l5}
\end {lemma}
\begin {proof}
As in the lemma \ref {l2}, it is proved that for any $ t \in T_\alpha $
the finite subgroup $ L_t = \langle a, a ^ t \rangle $ is contained in the appendix
of a finite Frobenius group $ M $ from $ \mathfrak {X} (L) $.
As in the lemma \ref {l3}, we conclude that either $ L_t = \langle a \rangle $,
or $ L_t $ is isomorphic to one of the groups $ SL_2 (3) $, $ SL_2 (5) $.
By the main theorem in \cite {Maz}, the normal closure
$ L = \langle a ^ {T_ \alpha} \rangle $ of $ a $ in $ T_ \alpha $ is locally finite.
As follows from the proof of the lemma \ref {l3}, or $ L = \langle a \rangle $,
or $ L $ is isomorphic to one of the groups $ SL_2 (3) $, $ SL_2 (5) $.
By lemma \ref {l4}, the theorem is true. The lemma is proved.
\end{proof}

\smallskip

A proper subgroup $ H $ of a group $ G $ is called {\it conjugate dense},
if $ H $ has a non-empty intersection with every conjugacy class in $ G $
elements.

\begin{lemma}
The theorem is true when $ T $ is a periodic group and
$ T_\alpha $ has no conjugate dense subgroups.
\label {l7}
\end {lemma}
\begin {proof}
By item 3, $ T_\alpha $ has a unique involution $ j $,
therefore, for an arbitrary element $ a $ of finite order from $ T_\alpha $
the subgroup $ L = \langle a, j \rangle $ is finite. Saturation condition
$ L \leq M \in \mathfrak {X} (L) $, and we can assume that
 $ H_M = L $, $ F_M \subseteq N_j $.
By item \ref {p7}, for any nonidentity element $ b \in N_j $, the subgroup
$ A = C_T (b) $ is periodic, contained in $ N_j $ and strongly isolated in $ T $,
in this case, $ N_T (A) = A \leftthreetimes H $, where $ H = T_\alpha \cap N_T (A) $.
Since in view of paragraphs. 2, 3 $ T_\alpha $ acts transitively by conjugation
on the set $ N_j $, then $ F_M ^ x \leq A $ for some $ x \in T_\alpha $.
This implies that $ a ^ x \in H $ and $ H $ is the conjugate dense subgroup
group $ T_\alpha $. According to the conditions of the lemma, $ H = T_\alpha $, by item 3, $ A = N_j $,
and the lemma is proved.
\end{proof}

\bigskip

We now complete the proof of the theorem. The first statement of the theorem follows from
Lemmas 3, 4. Statement 2 is proved in Lemma 6. Statement 3 of Theorem
coincides with Lemma 4. Statement 4 follows from Lemmas 2 and 5.
The theorem is proved.

\section{Proof of Theorem 2}

Let $ G $ be an infinite sharply triply transitive permutation group
multitudes
$ X = F \cup \{\infty \} $ satisfies the conditions of Theorem 2.
 As in \cite {D}, by $ B $ we denote the stabilizer $ G_\alpha $ of the point $ \alpha \in X $
and through $ H $ --- stabilizer $ G _{\alpha \, \beta} = G_\alpha \cap G_\beta $ of two
points
$ \alpha = \infty \in X $, $ \beta \in F $.
Let also $ J $ be the set of involutions of the group $ G $, and $ J_m $ be the set
involutions stabilizing exactly $ m $ points, $ m = 0, \, 1, \, 2 $.
\smallskip

\begin{proof}  of the theorem.
By the lemma, \ref {l1} $ B = U \leftthreetimes H $ is a Frobenius group,
$ H $ contains an involution $ z $ and $ N = N_G (H) = H \leftthreetimes \langle v \rangle $,
where $ v \in J $. For $ b \in U^\# $ there is an element $ a \in H $ of order
$ p-1 $ such that $ \langle a, b \rangle =
\langle b \rangle \leftthreetimes \langle a \rangle $ --- exactly twice transitive
group of order $ p (p-1) $. Let $ S $ be an arbitrary finite subgroup
from $ U $ containing an arbitrary element $ c $ from $ U \setminus \langle b \rangle $
and $ K = \langle b, S, a \rangle $. The subgroup $ K $ is obviously finite and by the saturation condition
$ K \leq M $, where $ M $ is a finite unsolvable group.
Let $ L = L (M) $ be the layer of the group $ M $ \cite [Proposition 1.4, p. 53] {Gor}.
It is clear that $ Z (L) = 1 $ and since the 2-rank of the group $ G $ is 2,
then $ L $ is a simple group. By virtue of lemma 6 \cite {D} $ L $ is isomorphic to $ L_2 (p^n) $ and
$ P = U \cap L $ is a Sylow $ p $ -subgroup of the group $ L $. As known,
all cyclic subgroups of $ P $ are conjugate in the subgroup
$ N_L (P) $, and $ P^\# $ splits into two conjugate classes.
Since $ \langle a, b \rangle \leq M $ and $ \langle a \rangle $ acts
transitively to $ \langle b \rangle ^ \# $, then obviously $ | M: L | = 2 $ and $ M \simeq PGL_2 (p ^ n) $.
The subgroup $ N_M (P) $ acts transitively on $ P^\# $, therefore
$ c = b^h $ for some $ h \in H \cap M $. Since the element $ c \in
P^\# $ we conclude that $ H $ is a periodic group.
By \cite [Theorem 2] {SD}, the group $ G $ is locally finite. The theorem is proved.
\end{proof}


\begin{thebibliography}{9}
\bibitem{BDS}
Durakov E.B., Bugaeva E.V., Sheveleva I.V.
On Sharply Doubly-Transitive Groups// Журн. Сиб. федер. ун-та.
Математика и физика,
Т.~6, 2013, C.~28-32.



\bibitem{SD3}
A. I. Sozutov, E. B. Durakov, On exactly doubly transitive groups with generalized
finite elements // Sib. math. jur., 58:5 (2017), 1144--1149
(http://www.mathnet.ru/rus/smj2925).


\bibitem{Kou}
Kourovka notebook: Unsolved problems of group theory --
6-17 ed., Novosibirsk, 1978-2012

\bibitem{RST} Eliyahu Rips, Yoav Segev, Katrin Tent,
  A sharply 2-transitive group without a non-trivial abelian normal subgroup//
arXiv:1406.0382v4 [math.GR], 22 Oct. 2014, 1-17.

\bibitem{RipsTent} Eliyahu Rips, Katrin Tent,
  Sharply 2-transitive groups in characteristic 0//
arXiv:1604.00573 [math.GR],  2 Apr 2016.


\bibitem{AndreTent} Simon Andre, Katrin Tent,
  Simple sharply 2-transitive groups//
arXiv: 2111.09580 [math.GR],  18 Nov 2021.

\bibitem{Turk} Seyfi Turkelli,
  Splitting of Sharply 2-Transitive Groups of Characteristic 3//
arXiv: 0809.0959 [math.GR],  5 Sep 2008.

\bibitem{TenZ}
Katrin Tent and Martin Ziegler,
Sharply 2-transitive groups// arXiv:1408.5612v1 [math.GR],
24 Aug 2014, 1-5.

\bibitem{J}
Jordan C. Recherches sur les substitutions//
J. Math. Pures Appl. (2), {\bf 17} (1872), 351 -- 363.


\bibitem{W}
H. W\"ahling, Theorie der Fastk\"orper. Essen: Thalen Ferlag.
1987.


\bibitem{Gor}
Gorenstein D. Finite simple groups.-- M.: Mir, 1985.


\bibitem{Hal}
M. Hall, Group Theory. M.: IL, 1962.

\bibitem{C}
A.I. Sozutov, On Shunkov groups acting freely on Abelian groups,
{\it Sib. math. journal}, {\bf 54}(2013), No. 1, 188-198.

\bibitem{C1}
A.I. Sozutov, On groups saturated with finite Frobenius groups,
{\it Mat. notes}, {\bf 54}(2013), no. 1, 188-198.

\bibitem{SD}
A.I. Sozutov, E.B.Durakov,
On the local finiteness of periodic exactly thrice transitive groups, {\it
Algebra and Logic}, {\bf 54}(2015), no. 1, 70-84.

\bibitem{Maz}
V.D. Mazurov, "Characterization of alternating groups", Algebra and logic,
{|bf 44:}1 (2005), 54-69.



\bibitem{Ten}
Tent K., Sharply 3-tranzitive groups// Advances in Mathematics
Volume 286, 2 January 2016, Pages 722-728




\bibitem{KraS}
Sozutov A. I., Kravtsova O. V. On $KT$-fields and exactly thrice transitive groups
// Algebra and Logic, 57 (2018), No. 2, 232-242.


\bibitem{PSS}
A.M. Popov, A.I. Sozutov, V.P. Shunkov, Groups with systems of Frobenius
subgroups//Krasnoyarsk: KSTU Publishing House. 2004. 210 pp. -- ISBN 5-7636-0654-X.

\bibitem{Preprint}
Sozutov A.I., Suchkov N.M. On some doubly transitive groups//
Krasnoyarsk.-- INM SO RAN.-- 1998.-- Preprint \No 17.-- P. 24.

\bibitem{CCC}
A.I. Sozutov, N.M. Suchkov, N.G. Suchkova, Infinite groups with involutions.--
Krasnoyarsk:
Siberian Federal University.-- 2011.-- 149 p.

\bibitem{D}
E. B. Durakov, Sharply 3-transitive Groups with Finite Element//Journal of Siberian Federal University. Mathematics \& Physics 2021, 14(3), 344–350


\end{thebibliography}
\end{document}